\documentclass[10pt, reqno]{amsart}
\usepackage{amsmath, amsthm, amscd, amsfonts, amssymb, graphicx, color}
\usepackage[bookmarksnumbered, colorlinks, plainpages]{hyperref}
\hypersetup{colorlinks=true,linkcolor=red, anchorcolor=green, citecolor=cyan, urlcolor=red, filecolor=magenta, pdftoolbar=true}
\usepackage{mathrsfs}
\usepackage[OT2,T1]{fontenc}
\DeclareSymbolFont{cyrletters}{OT2}{wncyr}{m}{n}
\DeclareMathSymbol{\Sha}{\mathalpha}{cyrletters}{"58}

\usepackage[OT2,OT1]{fontenc}
\newcommand\cyr
{
\renewcommand\rmdefault{wncyr}
\renewcommand\sfdefault{wncyss}
\renewcommand\encodingdefault{OT2}
\normalfont
\selectfont
}
\DeclareTextFontCommand{\textcyr}{\cyr}

\newtheorem{theorem}{Theorem}[section]
\newtheorem{lemma}[theorem]{Lemma}

\newtheorem{corollary}[theorem]{Corollary}
\theoremstyle{definition}
\newtheorem{definition}[theorem]{Definition}
\newtheorem{example}[theorem]{Example}

\theoremstyle{remark}
\newtheorem{remark}[theorem]{Remark}
\numberwithin{equation}{section}

\begin{document}
\setcounter{page}{1}

\title[Remark on height functions]{Remark on height functions}

\author[Nikolaev]
{Igor V. Nikolaev$^1$}

\address{$^{1}$ Department of Mathematics and Computer Science, St.~John's University, 8000 Utopia Parkway,  
New York,  NY 11439, United States.}
\email{\textcolor[rgb]{0.00,0.00,0.84}{igor.v.nikolaev@gmail.com}}

%%%%%%\dedicatory{All data are available as part of the manuscript}
%\dedicatory{In memory of Ola Bratteli}



%\date{Received:  August 14, 2015; Revised: yyyyyy; Accepted: zzzzzz.}

\begin{abstract}
Let $k$ be a number field and $V(k)$ an $n$-dimensional projective variety over $k$.
We use  the $K$-theory of a $C^*$-algebra $\mathcal{A}_V$ associated to $V(k)$
to define  a height of points of  $V(k)$.  The corresponding counting function  is calculated  
 and we show that it coincides with the known formulas for $n=1$.  
As an application, it is proved that  the set $V(k)$ is finite, whenever the sum of 
the odd Betti numbers of $V(k)$ exceeds $n+1$.   Our construction depends  on the 
$n$-dimensional  Minkowski question-mark function studied by Panti and others. 
\end{abstract}

\maketitle

%**************************************************************************
\section{Introduction}
%***************************************************************************
The height function is a measure of complexity of solutions 
of the diophantine equations.  For example, if $P^n(\mathbf{Q})$
is the $n$-dimensional projective space over $\mathbf{Q}$,  then the height of a
point $x\in P^n(\mathbf{Q})$ is given by the formula $H(x)=\max\{|x_0|,\dots,|x_n|\}$,
where $x_0,\dots,x_n\in\mathbf{Z}$ and $\mathbf{gcd}(x_0,\dots,x_n)=1$.
In general, if $V(\mathbf{Q})$ is a projective variety,  then  the height of
$x\in V(\mathbf{Q})$ is defined as  $H(x)=H(\phi(x))$,
where $\phi: V(\mathbf{Q})\hookrightarrow P^n(\mathbf{Q})$
is an embedding of $V(\mathbf{Q})$  into the  space 
$P^n(\mathbf{Q})$.

The fundamental property of $H(x)$ says that the number of points $x\in V(\mathbf{Q})$, 
such that $H(x)\le T$,  is  finite for any constant $T>0$.  Thus one can 
define  a counting function $\mathcal{N}(V(\mathbf{Q}), T)=\# \{x\in V(\mathbf{Q}) ~|~ H(x)\le T\}$. 
Such functions depend on the topology of the underlying variety. 
For instance,  if $X_g\subset  P^2(\mathbf{Q})$ is an algebraic curve of genus $g\ge 0$,  
then   (see [Hindry \& Silverman 2000] \cite[Theorem B.6.1 with $a=b=1$]{HS}):
%********************************************************************************************
\begin{equation}\label{eq1.1}
\mathcal{N}(X_g, ~T)\sim\begin{cases}
T, & \hbox{if} \quad g=0\cr 
\log ~T, & \hbox{if} \quad g=1\cr
Const, & \hbox{if} \quad g\ge 2.
\end{cases}
\end{equation}
%***********************************************************************
Clearly, if the counting function is bounded by a constant, then the set $V(\mathbf{Q})$
must be finite.  In particular, it follows from formulas (\ref{eq1.1}) that the set $X_g$
is finite,  whenever $g\ge 2$ (Faltings Finiteness Theorem).

Let $k$ be a number field and $V(k)$ an $n$-dimensional projective variety over $k$.
 Recall \cite[Section 5.3.1]{N} that the Serre $C^*$-algebra  $\mathcal{A}_V$
 is  the norm closure of a self-adjoint representation of the twisted 
 homogeneous coordinate ring of  $V(k)$  by the bounded linear operators acting on a Hilbert space, 
 see [Stafford \& van ~den ~Bergh 2001] \cite{StaVdb1} for the terminology.  
 Let $(K_0(\mathcal{A}_V), K_0^+(\mathcal{A}_V), \Sigma(\mathcal{A}_V))$ be the scaled dimension group 
 of the $C^*$-algebra $\mathcal{A}_V$ [Blackadar 1986] \cite[Section 6.1]{B}. 
  Minkowski question-mark function maps the scale $\Sigma(\mathcal{A}_V)\subset K_0(\mathcal{A}_V)$
 into a  subgroup of the additive torsion group $\mathbf{Q}^n/\mathbf{Z}^n$, see \cite[Lemma 8.4.1]{N}
 for $n=1$ and Lemma \ref{lm3.2}  for $n\ge 1$. 
  On the other hand, there exists a proper inclusion of the sets $V(k)\subseteq  K_0(\mathcal{A}_V)$
 unless the Shafarevich-Tate group $\Sha(V(k))$ is trivial (Lemma \ref{lm3.1}). 
In other words,  $V(k)\subseteq \mathbf{Q}^n/\mathbf{Z}^n$. Clearly, this inclusion can be used to   define the height  $\mathscr{H}(x)$
of  a point $x\in V(k)$, see Definition \ref{dfn1.1} for the details.  

The aim of our note is a formula for the counting function associated to the height $\mathscr{H}(x)$
on a projective variety $V(k)$ of dimension $n\ge 1$  (Theorem \ref{thm1.1}).
For $n=1$ such a formula coincides with (\ref{eq1.1}) up to a logarithm to the base $2$ (Corollary \ref{cor4.1}). 
 As an application, we prove that  the set $V(k)$ is finite,
whenever the sum of the odd Betti numbers of $V(k)$ exceeds 
$n+1$ (Corollary \ref{cor1.2}).  To formalize our results,  we introduce the following notation.

An $n$-dimensional Minkowski question-mark function $?^n: \mathbf{R}^n/\mathbf{Z}^n\to  \mathbf{R}^n/\mathbf{Z}^n$
is a one-to-one continuous function which maps: (i) $n$-tuples of rational numbers to such of the dyadic rationals;
(ii) $n$-tuples of algebraic numbers of degree $n+1$ over $\mathbf{Q}$ to such of the non-dyadic rational numbers;
and (iii) $n$-tuples of remaining irrational numbers to such of the irrational numbers [Minkowski 1904] \cite[case $n=1$]{Min1}
and [Panti 2008] \cite[case $n\ge 1$]{Pan1}.

Let $\tau$ be the canonical tracial state on the $C^*$-algebra $\mathcal{A}_V$ 
and let $\tau_*((K_0(\mathcal{A}_V))=\mathbf{Z}+\mathbf{Z}\theta_1+\dots+\mathbf{Z}\theta_n\subset\mathbf{R}$ 
be the value of $\tau$ on the $K_0$-group [Blackadar 1986] \cite[Section 6.9]{B}. 
It is known that $\theta_i$ are  algebraic numbers of degree $n+1$ over $\mathbf{Q}$ \cite[p. 186]{N}.
The Minkowski function $?^n$ maps the $\theta_i$ to  rational numbers $\frac{p_i}{q_i}$, i.e. 
$?^n(1,\theta_1,\dots,\theta_n)=(1,\frac{p_1}{q_1},\dots, \frac{p_n}{q_n})$. 
Since $(1,\frac{p_1}{q_1},\dots, \frac{p_n}{q_n})\in P^n(\mathbf{Q})$, 
one can define the height function $\mathscr{H}$ on $K_0(\mathcal{A}_V)$
as follows.

\smallskip
%***************************************************************
\begin{definition}\label{dfn1.1}
$\mathscr{H}(1,\theta_1,\dots,\theta_n):=H(1,\frac{p_1}{q_1},\dots, \frac{p_n}{q_n})$,
where $H$ is the standard height of points of the  space  $P^n(\mathbf{Q})$, 
see Section 1 or  Example \ref{ex2.1}.  
\end{definition}
%****************************************************

\smallskip
As explained, Lemma \ref{lm3.1} says that $V(k)\subseteq K_0(\mathcal{A}_V)$. 
Thus each point  $x\in V(k)$ get assigned a height $\mathscr{H}(x)$. 
The corresponding counting function is given by the formula $N(V(k), T)=\# \{x\in V(k) ~|~ \mathscr{H}(x)\le T\}$.
Denote by  $rk ~K_0(\mathcal{A}_V)$ the rank of the abelian group $K_0(\mathcal{A}_V)$. 
Our main results can be formulated as follows. 
%*********************************************
\begin{theorem}\label{thm1.1}
Let $V(k)$ be an $n$-dimensional projective variety over the number field $k$. Then
%********************************************************************************************
\begin{equation}\label{eq1.2}
\log_2 N(V(k), T)\sim\begin{cases}
T^n, & \hbox{if} \quad rk ~K_0(\mathcal{A}_V) < n+1\cr 
n ~\log_2 ~T, & \hbox{if} \quad  rk ~K_0(\mathcal{A}_V) = n+1\cr
Const, & \hbox{if} \quad rk ~K_0(\mathcal{A}_V) > n+1.
\end{cases}
\end{equation}
%***********************************************************************
\end{theorem}
%************************************************

\bigskip
Let $\beta_i$ be the $i$-th Betti number of the $n$-dimensional variety $V(k)$. 
Theorem \ref{thm1.1}  and the Chern character formula imply the following finiteness result. 

\bigskip
%**************************************************
\begin{corollary}\label{cor1.2}
The set $V(k)$ is finite,  whenever $\sum_{i=1}^{n} \beta_{2i-1}> n+1$. 
\end{corollary}
%***************************************************

\bigskip
The paper is organized as follows.  A brief review of the preliminary facts is 
given in Section 2. Theorem \ref{thm1.1} and Corollary \ref{cor1.2} are
proved in Section 3.  The case $n=1$ of Theorem \ref{thm1.1} is considered
in Section 4. 

%**************************************************************************
\section{Preliminaries}
%***************************************************************************
We briefly review the height and Minkowski question-mark functions,  and Serre $C^*$-algebras.  
We refer the reader to    [Hindry \& Silverman 2000] \cite[Part B]{HS},  
 [Minkowski 1904] \cite{Min1},  [Panti 2008] \cite{Pan1}, 
  [Stafford \& van ~den ~Bergh 2001] \cite{StaVdb1} 
  and \cite[Section 5.3.1]{N} for a detailed exposition.

%**************************************************************************
\subsection{Height functions}
%***************************************************************************
Let $k$ be a number field and let $V(k)$ be a smooth projective variety 
defined over $k$ with a fixed embedding $V\subset P^n$. A function 
$H:V(k)\to \mathbf{R}$ is called a height function, if for any constant $T$ the
set $\{x\in V(k) ~|~ H(x)\le T\}$ is finite. 
%******************************************
\begin{example}\label{ex2.1}
Let $P^n(\mathbf{Q})$ be an $n$-dimensional projective space over the 
rational numbers $\mathbf{Q}$. If $x\in P^n(\mathbf{Q})$, one can write
$x=(x_0, x_1,\dots, x_n)$, where $x_i\in \mathbf{Z}$ and $gcd~ (x_0, x_1,\dots, x_n)=1$.
The height of point $x$ is given by the formula:
%****************************************************************
\begin{equation}\label{eq2.2}
H(x)=\max \{|x_0|, |x_1|,\dots, |x_n|\}.
\end{equation}
%*************************************************************  
It is easy to see, that the set $\{x\in P^n(\mathbf{Q}) ~|~ H(x)\le T\}$
is finite, since there are only finitely many integers whose absolute value
is less of equal to the constant $T$. 
\end{example}
%*****************************************
%*****************************************************
\begin{definition}
Given the height function $H: V(k)\to \mathbf{R}$, the counting function is 
defined as 
 %****************************************************************
\begin{equation}
\mathcal{N}(V(k), T)=\# \{x\in V(k) ~|~ H(x)\le T\}. 
\end{equation}
%*************************************************************  
\end{definition}
%****************************************************
The following theorem gives a formula of the counting function 
for varieties  of dimension  $n=1$. 
%************************************************************
\begin{theorem}
{\bf (\cite[Th. B.6.1]{HS})}
Let $k$ be a number field and let $X_g(k)$ be a smooth projective curve
of genus $g$ over $k$. Assume that $X_g(k)$ is not empty. Then there are constants
$a$ and $b$, which depend on $X_g(k)$ and on the height used in the counting 
function, such that 
%********************************************************************************************
\begin{equation}\label{eq2.4}
\mathcal{N}(X_g(k), ~T)\sim\begin{cases}
aT^b, & \hbox{if} \quad g=0 ~(a,b>0)\cr 
a(\log ~T)^b, & \hbox{if} \quad g=1 ~(a>0,b\ge 0)\cr
a, & \hbox{if} \quad g\ge 2.
\end{cases}
\end{equation}
%***********************************************************************
\end{theorem}
%************************************************************

%**************************************************************************
\subsection{Minkowski question-mark function}
%***************************************************************************
Minkowski question-mark function is defined by the convergent
series
 %***********************************************************************************************
 \begin{equation}\label{eq2.5} 
 ?(x):=a_0+2\sum_{k=1}^{\infty} \frac{(-1)^{k+1}}{2^{a_1+\dots+a_k}},
 \end{equation}
 %******************************************************************************************** 
where $x=[a_0,a_1,a_2,\dots]$ is the  continued fraction of the irrational number $x$. 
The $?(x): [0,1]\to [0,1]$ is a monotone continuous function with the following properties
 [Minkowski 1904] \cite[p. 172]{Min1}:

\medskip
(i) $?(0)=0$ and $?(1)=1$; 

\smallskip
(ii)  $?(\mathbf{Q})=\mathbf{Z}[\frac{1}{2}]$ are dyadic rationals;

\smallskip
(iii) $?(\mathscr{Q})=\mathbf{Q}-\mathbf{Z}[\frac{1}{2}]$, where $\mathscr{Q}$ are quadratic 
irrational numbers.

\bigskip
An $n$-dimensional  generalization of properties (i)-(iii) is as follows. 
%***************************************************************
\begin{theorem}\label{thm2.4}
{\bf (\cite[Theorem 2.1]{Pan1})}
There exists a unique $n$-dimensional Minkowski question-mark function $?^n: \mathbf{R}^n/\mathbf{Z}^n\to  \mathbf{R}^n/\mathbf{Z}^n$
which is one-to-one,  continuous and maps: 

\medskip
(i) $n$-tuples of the rational numbers to such of the dyadic rationals;

\smallskip
(ii) $n$-tuples of the algebraic numbers of degree $n+1$ over $\mathbf{Q}$ to such of the non-dyadic rational numbers;

\smallskip
(iii) $n$-tuples of remaining irrational numbers to such of the irrational numbers.
\end{theorem}
%***************************************************************
 %**************************************************************
 \begin{remark}
 Panti's Theorem \cite[Theorem 2.1]{Pan1} was stated in terms of a unique homeomorphism 
 $\Phi:\mathbf{R}^n/\mathbf{Z}^n\to  \mathbf{R}^n/\mathbf{Z}^n$
conjugating  the tent map $T$ and  the M\"onkemeyer map $M$  [Panti 2008] \cite[Section 2]{Pan1}. 
 The reader can verify that the maps $T$ and $M$ characterize a unique  $n$-dimensional Minkowski 
 question-mark function $?^n: \mathbf{R}^n/\mathbf{Z}^n\to  \mathbf{R}^n/\mathbf{Z}^n$ satisfying properties (i)-(iii)
 of Theorem \ref{thm2.4}. 
 \end{remark} 
%*************************************************************

%**************************************************************************
\subsection{Serre $C^*$-algebras}
%***************************************************************************
Let $V$ be a projective variety over the field $k$.  Denote by $\mathcal{L}$ an invertible
sheaf of the linear forms on $V$.  If $\sigma$ is an automorphism of $V$,  then
the pullback of $\mathcal{L}$ along $\sigma$ will be denoted by $\mathcal{L}^{\sigma}$,
i.e. $\mathcal{L}^{\sigma}(U):= \mathcal{L}(\sigma U)$ for every $U\subset V$. 
The graded $k$-algebra
 %*************************************************************************
%\begin{equation}\label{eq2.1}
$B(V, \mathcal{L}, \sigma)=\bigoplus_{i\ge 0} H^0\left(V, ~\mathcal{L}\otimes \mathcal{L}^{\sigma}\otimes\dots
\otimes  \mathcal{L}^{\sigma^{ i-1}}\right)$
%\end{equation}
%*************************************************************************  
is called a  twisted homogeneous coordinate ring of $V$ [Stafford \& van den Bergh 2001]  \cite{StaVdb1}.  Such a ring is 
always non-commutative,  unless the automorphism $\sigma$ is trivial. 
A multiplication of sections of  $B(V, \mathcal{L}, \sigma)=\oplus_{i=1}^{\infty} B_i$ is defined by the 
rule  $ab=a\otimes b$,   where $a\in B_m$ and $b\in B_n$.
An invertible sheaf $\mathcal{L}$ on $V$  is called $\sigma$-ample, if for 
every coherent sheaf $\mathcal{F}$ on $V$,
 the cohomology group $H^k(V, ~\mathcal{L}\otimes \mathcal{L}^{\sigma}\otimes\dots
\otimes  \mathcal{L}^{\sigma^{ n-1}}\otimes \mathcal{F})$  vanishes for $k>0$ and
$n>>0$.   If $\mathcal{L}$ is a $\sigma$-ample invertible sheaf on $V$,  then
%************************************************************************************
%\begin{equation}\label{eq2.2}
$Mod~(B(V, \mathcal{L}, \sigma)) / ~Tors ~\cong ~Coh~(V)$,
%\end{equation}
%****************************************************************************
where  $Mod$ is the category of graded left modules over the ring $B(V, \mathcal{L}, \sigma)$,
$Tors$ is the full subcategory of $Mod$ of the torsion  modules and  $Coh$ is the category of 
quasi-coherent sheaves on a scheme $V$.  In other words, the $B(V, \mathcal{L}, \sigma)$  is  
a coordinate ring of the variety $V$.

Let $R$ be a commutative  graded ring,  such that $V=Proj~(R)$.  
Denote by $R[t,t^{-1}; \sigma]$
the ring of skew Laurent polynomials defined by the commutation relation
$b^{\sigma}t=tb$  for all $b\in R$, where $b^{\sigma}$ is the image of  $b$ under automorphism 
$\sigma$.  It is known, that $R[t,t^{-1}; \sigma]\cong B(V, \mathcal{L}, \sigma)$ [Stafford \& van den Bergh 2001]  \cite[Section 5]{StaVdb1}.
Let $\mathcal{H}$ be a Hilbert space and   $\mathscr{B}(\mathcal{H})$ the algebra of 
all  bounded linear  operators on  $\mathcal{H}$.
For a  ring of skew Laurent polynomials $R[t, t^{-1};  \sigma]$,  
 consider a homomorphism
$\rho: R[t, t^{-1};  \sigma]\longrightarrow \mathscr{B}(\mathcal{H})$. 
Recall  that  $\mathscr{B}(\mathcal{H})$ is endowed  with a $\ast$-involution;
the involution comes from the scalar product on the Hilbert space $\mathcal{H}$. 
The representation $\rho$ is called  $\ast$-coherent,   if
(i)  $\rho(t)$ and $\rho(t^{-1})$ are unitary operators,  such that
$\rho^*(t)=\rho(t^{-1})$ and 
(ii) for all $b\in R$ it holds $(\rho^*(b))^{\sigma(\rho)}=\rho^*(b^{\sigma})$, 
where $\sigma(\rho)$ is an automorphism of  $\rho(R)$  induced by $\sigma$. 
Whenever  $B=R[t, t^{-1};  \sigma]$  admits a $\ast$-coherent representation,
$\rho(B)$ is a $\ast$-algebra.  The norm closure of  $\rho(B)$  is   a   $C^*$-algebra
   denoted  by $\mathcal{A}_V$.  We  refer to  $\mathcal{A}_V$  as   the    Serre $C^*$-algebra
 of  $V$ \cite[Section 5.3.1]{N}. 

Let $(K_0(\mathcal{A}_V), K_0^+(\mathcal{A}_V), u)$
%*****************************************************************
\footnote{The order unit $u$ is equivalent to defining a  scale $\Sigma(\mathcal{A}_V)$ 
in the dimension group, i.e. the generating set of its positive cone  $K_0^+(\mathcal{A}_V)$.} 
%***************************************************************
 be the scaled dimension group  of the 
$C^*$-algebra $\mathcal{A}_V$ [Blackadar 1986] \cite[Sections 6.3 and 7.3]{B}. 
Such a group defines the Serre $C^*$-algebra $\mathcal{A}_V$ up to an isomorphism. 
Moreover, if $V$ is defined over a number field,  then its dimension group is order-isomorphic
to the dimension group $(O_K, O_K^+, [u])$, where  $O_K$ is the ring of integers of a number field $K$,
$O_K^+$ the additive semigroup of positive integers for a real embedding 
$O_K\hookrightarrow\mathbf{R}$ and $[u]$  is a unit of  $O_K$ \cite[pp. 194-195]{N}. 
Such a dimension group is called  stationary.

%**************************************************************************
\section{Proofs}
%***************************************************************************

%**************************************************************************
\subsection{Proof of theorem \ref{thm1.1}}
%***************************************************************************
For the sake of clarity, let us outline the main ideas. 
We start with a preparatory lemma saying that the Minkowski question-mark function
defines a functor from a category of the stationary  scaled dimension groups
 $(K_0(\mathcal{A}_V), K_0^+(\mathcal{A}_V), \Sigma(\mathcal{A}_V))$
  to such of the infinite subgroups of the additive torsion group 
 $\mathbf{Q}^n/\mathbf{Z}^n$ (Lemma \ref{lm3.2}).  
 This purely algebraic fact is interesting on its own. 
 Back to theorem \ref{thm1.1}, 
our proof hinges  on the inclusion of the sets 
$V(k)\subseteq  K_0(\mathcal{A}_V)$ (Lemma \ref{lm3.1}). 
Notice that the sets $V(k)$ and $K_0(\mathcal{A}_V)$
are one-to-one, if and only if,  the Shafarevich-Tate group of $V(k)$ is trivial,
i.e. $V(k)$ satisfies the local-global (Hasse) principle. 
 The rest of the proof follows from an elementary analysis of the
 counting function $ N(V(k), T)$ (Lemma \ref{lm3.3}). 
 Let us pass to a detailed argument.

%***********************************************
\begin{lemma}\label{lm3.2}
The $n$-dimensional Minkowski question-mark function 
\linebreak
$?^n(x): \mathbf{R}^n/\mathbf{Z}^n\to \mathbf{R}^n/\mathbf{Z}^n$
defines a functor $F$ from a  category of the stationary  scaled dimension groups
 $(K_0(\mathcal{A}_V), K_0^+(\mathcal{A}_V), \Sigma(\mathcal{A}_V))$
  to such of the infinite subgroups of the  torsion abelian group 
 $\mathbf{Q}^n/\mathbf{Z}^n$,  such that $F$ maps order-isomorphic scaled dimension groups
 to the isomorphic infinite torsion abelian groups.  
\end{lemma} 
%***********************************************
\begin{proof}
(i) Let  $?^n(x): \mathbf{R}^n/\mathbf{Z}^n\to \mathbf{R}^n/\mathbf{Z}^n$ be the $n$-dimensional 
Minkowski question-mark function constructed in Theorem \ref{thm2.4}. 
Consider a scaled dimension group  $(K_0(\mathcal{A}_V), K_0^+(\mathcal{A}_V), \Sigma(\mathcal{A}_V))$
of the Serre $C^*$-algebra $\mathcal{A}_V$ of the $n$-dimensional projective variety $V(k)$. 
Let $\tau: \mathcal{A}_V\to\mathbf{C}$ be the canonical trace on the $C^*$-algebra  $\mathcal{A}_V$
and $\tau_*$ its value on the group  $K_0(\mathcal{A}_V)$. 
Since the dimension group  $(K_0(\mathcal{A}_V), K_0^+(\mathcal{A}_V))$ is of a stationary
type \cite[p. 47]{N},  we conclude that
%**********************************************************************************
\begin{equation}\label{eq3.1}
\tau_*((K_0(\mathcal{A}_V))=\mathbf{Z}+\mathbf{Z}\theta_1+\dots+\mathbf{Z}\theta_n\subset\mathbf{R},
\end{equation}
%******************************************************************************
where $\theta_i$ are real algebraic numbers of degree $n+1$ over $\mathbf{Q}$ \cite[p. 186]{N}.
One can apply to $\theta_i$ item (ii) of Theorem \ref{thm2.4} and thus obtain the non-dyadic rationals
$p_i/q_i =?^n(\theta_i)$.  
If $\Lambda\subset K$ is the $\mathbf{Z}$-module (\ref{eq3.1}) lying 
in the number field $K:=\mathbf{Q}(\theta_i)$, then its image $?^n(\Lambda)$ 
under the Minkowski question-mark function is an infinite torsion abelian group $G$
given by the formula:
%**********************************************************************************
\begin{equation}\label{eq3.2}
G:=?^n(\Lambda)=\bigoplus_{\mathbf{Z}}\left(\frac{p_1^{\mathbf{Z}}}{q_1^{\mathbf{Z}}},\dots, \frac{p_n^{\mathbf{Z}}}{q_n^{\mathbf{Z}}}\right)\subset \mathbf{Q}^n/\mathbf{Z}^n,
\end{equation}
%******************************************************************************
where $p_i^{\mathbf{Z}}/q_i^{\mathbf{Z}} ~:= ~?^n(\mathbf{Z}\theta_i)$. 
Thus one gets a map $F$ on the set of all stationary  scaled dimension groups
 of rank $n+1$  acting by the formula 
 $(K_0(\mathcal{A}_V), K_0^+(\mathcal{A}_V), \Sigma(\mathcal{A}_V))\mapsto G\subset \mathbf{Q}^n/\mathbf{Z}^n$. 

\medskip
(ii) Let us prove that $F$  maps order-isomorphic scaled  dimension groups
 to the isomorphic infinite torsion abelian groups. In other words,
  the set $?^n(\Lambda)$ does not depend 
 on the basis $(\theta_1,\dots,\theta_n)$ of the $\mathbf{Z}$-module
 $\Lambda$. 
 For the sake of clarity,  let us  consider the case $n=1$ first. 
 Let $\left(\small\begin{matrix} a & b\cr c & d\end{matrix}\right)\in SL_2(\mathbf{Z})$
be a change of basis in the $\mathbf{Z}$-module $\Lambda=\mathbf{Z}+\mathbf{Z}\theta$. 
Then $\Lambda'=\mathbf{Z}(a+b\theta)+\mathbf{Z}(c+d\theta)$ and scaling by 
$c+d\theta$ gives us $\Lambda'=\mathbf{Z}+\mathbf{Z}\theta'$, where 
$\theta'=\frac{a+b\theta}{c+d\theta}$. 
Since $\theta'$ belongs to the orbit of $\theta$ under the action of the modular group $SL_2(\mathbf{Z})$,
the continued fractions of $\theta$ and $\theta'$ must coincide except possibly in a finite number of 
terms, i.e.  $\theta=[a_0,a_1,a_2,\dots]$ and  $\theta'=[b_1,\dots b_N; a_0,a_1,a_2,\dots]$ for some $N>0$. 
The values of the Minkowski question-mark function (\ref{eq2.5}) on $\theta$ and $\theta'$ is given 
by the formula:
%**********************************************************************************
\begin{equation}\label{eq3.3}
?(\theta')=?(\theta) + 2\sum_{k=1}^N \frac{(-1)^{k+1}}{2^{b_1+\dots+b_k}}. 
\end{equation}
%******************************************************************************
Since $\theta$ and $\theta'$ are quadratic irrationals, the values of $?(\theta)$ and $?(\theta')$ are non-dyadic 
rationals. Moreover, from (\ref{eq3.3}) the difference $?(\theta')-?(\theta)=2\sum_{k=1}^N \frac{(-1)^{k+1}}{2^{b_1+\dots+b_k}}$ is a dyadic
rational. Therefore the set $?(\Lambda)$ consists of the non-dyadic rational $?(x)$ and  sums of $?(x)$ with all dyadic rationals. 
Clearly, the set $?(\Lambda)$ is an algebraically closed subset of the set of non-dyadic rationals and   $?(\Lambda)$ is invariant of the 
change of basis in the $\mathbf{Z}$-module $\Lambda$. The values of $?(\Lambda)$ is a generating set of the group $G$ in formula (\ref{eq3.2})
with $n=1$.   The case $n\ge 2$ is treated likewise using the $n$-dimensional continued fractions  [Bernstein 1971] \cite{BE}. 
Lemma \ref{lm3.2} is proved. 
\end{proof}
%*********************************************

%***********************************************
\begin{lemma}\label{lm3.1}
There exists an inclusion of the sets $V(k)\subseteq  K_0(\mathcal{A}_V)$
which is proper unless the Shafarevich-Tate group $\Sha(V(k))=1$. Moreover, if $V(k)$ is an abelian 
variety, there exists a group homomorphism $V(k)/V_{tors}(k)\to K_0(\mathcal{A}_V)$,
where $V_{tors}(k)$ is the torsion subgroup of $V(k)$. 
\end{lemma} 
%***********************************************
\begin{proof}  
Roughly speaking,  Lemma \ref{lm3.1} follows from the local-global (Hasse)
principle for the variety $V(k)$ followed by  a localization formula of the Serre
$C^*$-algebra $\mathcal{A}_V$ \cite[Section 6.7]{N}. We shall split the proof 
in several steps.

\smallskip
(i) Let $\mathbf{F}_p$ be a finite field and let $V(\mathbf{F}_p)$ be the reduction of 
$V(k)$ modulo the prime ideal $\mathscr{P}\subset k$ over $p$. 
Recall that if $x\in V(k)$ is a point, then there exists a point $x_p\in V(\mathbf{F}_p)$
for very prime $p$. Whenever the converse is true including a solution over the reals, 
the variety $V(k)$ is said to satisfy the local-global (Hasse) principle.  
It is known that $V(k)$ satisfies such a principle if and only if the Shafarevich-Tate 
group $\Sha(V(k))$ of $V(k)$ is trivial  [Hindry \& Silverman 2000] \cite[p. 75]{HS}. 

\medskip
(ii)  One can recast the Hasse principle as follows. Consider a direct sum $\oplus_p V(\mathbf{F}_p)$
of the finite sets $V(\mathbf{F}_p)$.  Each element $(x_{p_1}, x_{p_2}, \dots)\in \oplus_p V(\mathbf{F}_p)$
is the system of local points $x_{p_i}\in V(\mathbf{F}_{p_i})$ obtained by the reduction of a global
point $x\in V(k)$. The latter exists if and only if  $\Sha(V(k))=1$ and does not exist otherwise. Thus 
one gets an inclusion of the sets:
%**********************************************************************************
\begin{equation}\label{eq3.4}
V(k)\subseteq \bigoplus_p V(\mathbf{F}_p). 
\end{equation}
%******************************************************************************

\medskip
(iii) We shall  use the following localization result for the Serre $C^*$-algebras   $\mathcal{A}_V$ 
 \cite[Section 6.7]{N}. Let $\varepsilon_i^{\pi_i(p)}$ be the shift automorphism of a lattice 
 $\Lambda_i^p$ defined via the $K$-theory of the $C^*$-algebra $\mathcal{A}_V$;
 we refer the reader to   \cite[p. 195]{N} for the notation and details. The following relation between 
 the number $|V(\mathbf{F}_p|$ and the $K_0$-group of a crossed product of   $\mathcal{A}_V$
  is true \cite[Theorem 6.7.1]{N}
  \footnote{The original statement is given in terms of the traces of  $\varepsilon_i^{\pi_i(p)}$ but equivalent to formula (\ref{eq3.5}).}
  :
%**********************************************************************************
\begin{equation}\label{eq3.5}
|V(\mathbf{F}_p)|= \left| K_0(\mathcal{A}_V\rtimes_{\oplus_{i=0}^{2n} \varepsilon_i^{\pi_i(p)}}\mathbf{Z})\right|,
\end{equation}
%******************************************************************************
 where $\oplus_{i=0}^{2n} \varepsilon_i^{\pi_i(p)}$ is an endomorphism of $\mathcal{A}_V$
 corresponding to the direct sum of $\varepsilon_i^{\pi_i(p)}$.

\medskip
(iv) Formula (\ref{eq3.5}) endows the set $V(\mathbf{F}_p)$ with the structure of a finite abelian group.
Thus the RHS of (\ref{eq3.4}) can be written as the limit of finite abelian groups:
%**********************************************************************************
\begin{equation}\label{eq3.6}
\lim_{m\to\infty} \bigoplus_{i=1}^m V(\mathbf{F}_{p_i}). 
\end{equation}
%******************************************************************************
Since (\ref{eq3.6}) consists of the direct sums of finite abelian groups,
it converges in the discrete topology to a torsion abelian group $G$. 

\medskip
(v) On the other hand, the scaled dimension group
 $(K_0(\mathcal{A}_V), K_0^+(\mathcal{A}_V), \Sigma(\mathcal{A}_V))$
 gives rise to an infinite torsion abelian group $G'$ (Lemma \ref{lm3.2}). 
It is not hard to see that $G'\cong G$, since both groups depend
on $\mathcal{A}_V$ functorially. In particular, one gets
via the Minkowski question-mark function a bijective correspondence
between the element of the groups $G=\oplus_p V(\mathbf{F}_p)$
and  $K_0(\mathcal{A}_V)$. In view of (\ref{eq3.4}),  there exists
an inclusion of the sets: 
%**********************************************************************************
\begin{equation}\label{eq3.7}
V(k)\subseteq  K_0(\mathcal{A}_V). 
\end{equation}
%******************************************************************************
Clearly, the inclusion (\ref{eq3.7}) is proper unless the Shafarevich-Tate group of $V(k)$
is trivial.

\medskip
(vi)  Let us show that if $V(k)$ is an abelian 
variety,  then there exists a group homomorphism 
$h: V_0(k) \to K_0(\mathcal{A}_V)$,
where $V_0(k):=V(k)/V_{tors}(k)$. 
Denote by $\overline{V_0(k)}$ the algebraic closure
of the set $V_0(k)\subseteq  K_0(\mathcal{A}_V)$
and  consider  the following two cases:

\medskip
{\bf Case I.}  $\overline{V_0(k)}=V_0(k)$. In this case 
$rk ~V_0(k)\le rk  ~K_0(\mathcal{A}_V)$
and the homomorphism $h: V_0(k) \to  K_0(\mathcal{A}_V)$
is injective, i.e. $V_0(k)$ is either a subgroup of the abelian group $K_0(\mathcal{A}_V)$
or $V_0(k)\cong K_0(\mathcal{A}_V)$. 

\smallskip
{\bf Case II.}  $\overline{V_0(k)}=K_0(\mathcal{A}_V)$. 
 In this case  $rk ~V_0(k) > rk  ~K_0(\mathcal{A}_V)$
 and  the homomorphism $h: V_0(k) \to  K_0(\mathcal{A}_V)$
is surjective with the $ker~h=\mathbf{Z}^{k - l}$,
where $k=rk ~V_0(k)$ and $l=rk  ~K_0(\mathcal{A}_V)$.  

\medskip
Thus in  both cases   one gets the required group homomorphism  $h: V_0(k) \to K_0(\mathcal{A}_V)$.
Lemma \ref{lm3.1} is proved. 
\end{proof}
%*********************************************

%***********************************************
\begin{lemma}\label{lm3.3}
Let $N(V(k), T)=\# \{x\in V(k) ~|~ \mathscr{H}(x)\le T\}$.
Then
%********************************************************************************************
\begin{equation}\label{eq3.8}
\log_2 N(V(k), T)\sim\begin{cases}
T^n, & \hbox{if} \quad rk ~K_0(\mathcal{A}_V) < n+1\cr 
n ~\log_2 ~T, & \hbox{if} \quad  rk ~K_0(\mathcal{A}_V) = n+1\cr
Const, & \hbox{if} \quad rk ~K_0(\mathcal{A}_V) > n+1.
\end{cases}
\end{equation}
%***********************************************************************
\end{lemma} 
%***********************************************
\begin{proof}
Roughly speaking, formulas (\ref{eq3.8}) follow from 
an elementary analysis of the $n$-dimensional Minkowski question-mark function (Theorem \ref{thm2.4}). 
Namely, let $n=\dim_{\mathbf{C}}V(k)$.  We split the proof in the following steps. 

\medskip
{\bf Case I: $rk ~K_0(\mathcal{A}_V) < n+1$}. 
(i)  Let $\tau_*((K_0(\mathcal{A}_V))=\mathbf{Z}+\mathbf{Z}\theta_1+\dots+\mathbf{Z}\theta_n\subset\mathbf{R}$
and consider the number field $K=\mathbf{Q}(\theta_i)$. The field $K$ is a  Galois extension
of  $\mathbf{Q}$ of degree $n+1$ \cite[Remark 6.6.1]{N} and thus  $K$ is a totally real number field. The $\mathbf{Z}$-module   
$\tau_*((K_0(\mathcal{A}_V))$ itself is independent of the embeddings of $K$ into $\mathbf{R}$. 
In particular, the generators $\theta_i$ of $\tau_*((K_0(\mathcal{A}_V))$ are conjugate algebraic numbers under action of the Galois
group $Gal~(K|\mathbf{Q})$. 

\smallskip
(ii) It is not hard to see that when $rk ~K_0(\mathcal{A}_V)=k < n+1$, 
the $\mathbf{Z}$-module $\tau_*((K_0(\mathcal{A}_V))$ is invariant
of the embedding if and only if $\theta_i:=\frac{p_i}{q_i}\in\mathbf{Q}$.
Indeed, since $|Gal~(K|\mathbf{Q})|=n+1$
we conclude that  $\tau_*((K_0(\mathcal{A}_V))=\mathbf{Z}+\mathbf{Z}\theta_1+\dots+\mathbf{Z}\theta_k$ is fixed by the action of
$Gal~(K|\mathbf{Q})$ if and only if  $\theta_i:=\frac{p_i}{q_i}\in\mathbf{Q}$,  i.e. $rk ~K_0(\mathcal{A}_V)=1$.

\smallskip
(iii)  We shall evaluate the Minkowski question-mark function $?^n(x): \mathbf{R}^n/\mathbf{Z}^n
\to \mathbf{R}^n/\mathbf{Z}^n$ at the point 
$x=(\frac{p_1}{q_1},\dots, \frac{p_n}{q_n})$. 
By item (i) of Theorem \ref{thm2.4}, one gets:
%**********************************************************************************
\begin{equation}\label{eq3.9}
?^n\left(\frac{p_1}{q_1},\dots, \frac{p_n}{q_n}\right)=\left(\frac{l_1}{2^{k_1}},\dots, \frac{l_n}{2^{k_n}}\right)
\in  \mathbf{R}^n/\mathbf{Z}^n,
\end{equation}
%******************************************************************************
where $0\le l_i\le 2^{k_i}$.  Recall from   Definition \ref{dfn1.1} that 
%**********************************************************************************
\begin{equation}\label{eq3.10}
\mathscr{H}(1,\theta_1,\dots,\theta_n):=H\left(1,\frac{p_1}{q_1},\dots, \frac{p_n}{q_n}\right),
\end{equation}
%******************************************************************************
where $H$ is the height function on $P^n(\mathbf{Q})$ given by formula (\ref{eq2.2}).
Clearing denominators, one gets
%**********************************************************************************
\begin{equation}\label{eq3.11}
H\left(1,\frac{p_1}{q_1},\dots, \frac{p_n}{q_n}\right)=\max ~\{|q_1\dots q_n|, |p_1Q_1|,\dots, |p_nQ_n|\},
\end{equation}
%******************************************************************************
where $Q_i=q_1\dots q_{i-1}q_{i+1}\dots q_n$.  Returning to our case given by formula (\ref{eq3.9}),
one obtains
%**********************************************************************************
\begin{equation}\label{eq3.12}
H\left(1, \frac{l_1}{2^{k_1}},\dots, \frac{l_n}{2^{k_n}}\right)=2^{k_1}\dots 2^{k_n}. 
\end{equation}
%******************************************************************************

\smallskip
(iv)  The variable $T$ will be chosen as follows. 
Notice that function $H$ in formula (\ref{eq3.12}) depends on the $n$ variables $k_i$.
We assume $k_i=T$ and using the counting function
$N(V(k), T)=\# \{x\in V(k) ~|~ H(x)\le T\}$, we conclude that $N(V(k), T)$ 
is equal to the number of points $k\le k_i=T$ in the $n$-cube $T^n$.
Such a number is proportional to the volume of the $n$-cube,  i.e. $N(V(k), T)\sim 2^{T^n}$. 
Taking the base $2$ logarithm on both sides of the estimate, 
one gets: 
%**********************************************************************************
\begin{equation}\label{eq3.13}
\log_2 N(V(k), T)\sim T^n. 
\end{equation}
%******************************************************************************
Case I of Lemma \ref{lm3.3} is proved.

\bigskip
{\bf Case II: $rk ~K_0(\mathcal{A}_V) = n+1$}. 
(i) Consider the $\mathbf{Z}$-module  $\tau_*((K_0(\mathcal{A}_V))=\mathbf{Z}+\mathbf{Z}\theta_1+\dots+\mathbf{Z}\theta_n\subset\mathbf{R}$
and let  $K=\mathbf{Q}(\theta_i)$ be a number field. Since the rank of $\tau_*((K_0(\mathcal{A}_V))$ is equal to $n+1$,
we conclude that $\deg (K|\mathbf{Q})=n+1$.  

\smallskip
(ii) Let us evaluate the Minkowski question-mark function $?^n(x): \mathbf{R}^n/\mathbf{Z}^n
\to \mathbf{R}^n/\mathbf{Z}^n$ at the point 
$x=(\theta_1,\dots, \theta_n)$.  In view of item (ii) of Theorem \ref{thm2.4}, one gets:
%**********************************************************************************
\begin{equation}\label{eq3.14}
?^n\left(\theta_1,\dots, \theta_n\right)=\left(\frac{p_1}{q_1},\dots, \frac{p_n}{q_n}\right)
\in  \mathbf{R}^n/\mathbf{Z}^n,
\end{equation}
%******************************************************************************
where $p_i/q_i$ are non-dyadic rational numbers and  $0\le p_i\le q_i$.  
Repeating the argument of  Case I, we use formula (\ref{eq3.11}) to calculate the height
function:
%**********************************************************************************
\begin{equation}\label{eq3.15}
H\left(1, \frac{p_1}{q_1},\dots, \frac{p_n}{q_n}\right)=q_1\dots q_n. 
\end{equation}
%******************************************************************************

\smallskip
(iii) Let us choose the variable $T$. 
Observe that  $H$ in formula (\ref{eq3.15}) depends on the $n$ variables $q_i$.
We let $q_i=T$ and using the counting function
$N(V(k), T)=\# \{x\in V(k) ~|~ H(x)\le T\}$, we conclude that $N(V(k), T)$ 
is equal to the number of points $q\le q_i=T$ in the $n$-cube $T^n$.
Such a number is proportional to the volume of the $n$-cube,  i.e. $N(V(k), T)\sim T^n$. 
Taking the base $2$ logarithm on both sides of the estimate, 
one gets: 
%**********************************************************************************
\begin{equation}\label{eq3.16}
\log_2 N(V(k), T)\sim n ~\log_2 ~T.  
\end{equation}
%******************************************************************************
Case II of Lemma \ref{lm3.3} is proved.

\bigskip
{\bf Case III: $rk ~K_0(\mathcal{A}_V) > n+1$}. 
 Roughly speaking, our proof of  the beautiful and important formula $\log_2 N(V(k), T)\sim Const$
 hinges on the almost everywhere differentiability
 of the Minkowski question-mark function $?^n(x)$.   Namely,   $\frac{d}{dx} ~?^n(x)\equiv0$ almost everywhere in  
 $\mathbf{R}^n/\mathbf{Z}^n$, see  [Dushistova \& Moshchevitin 2010] \cite[case $n=1$]{DusMos1} and Section 3 of [Panti 2008] \cite[case $n\ge 2$]{Pan1}. 
 Notice that restricting to $n=1$,  one gets a new proof of the Faltings Finiteness Theorem included in  formulas (\ref{eq1.1}). 
  Let us pass to a detailed argument.

(i) Again, we consider the $\mathbf{Z}$-module  $\tau_*((K_0(\mathcal{A}_V))=\mathbf{Z}+\mathbf{Z}\theta_1+\dots+\mathbf{Z}\theta_n\subset\mathbf{R}$
and let  $K=\mathbf{Q}(\theta_i)$ be a number field. Since $rk~\tau_*((K_0(\mathcal{A}_V))>n+1$,
one gets $\deg (K|\mathbf{Q})>n+1$.   Recall that $K$ is a  totally real Galois extension
of  the field $\mathbf{Q}$.  The $\mathbf{Z}$-module   
$\tau_*((K_0(\mathcal{A}_V))$ itself is independent of the embeddings of $K$ into $\mathbf{R}$. 
Such a condition is satisfied whenever $\theta_i$  are conjugate algebraic numbers  of degree higher than $n+1$.

\smallskip
(ii) We evaluate the Minkowski question-mark function $?^n(x): \mathbf{R}^n/\mathbf{Z}^n
\to \mathbf{R}^n/\mathbf{Z}^n$ at the point 
$x=(\theta_1,\dots, \theta_n)$.  In view of item (iii) of Theorem \ref{thm2.4}, one gets:
%**********************************************************************************
\begin{equation}\label{eq3.17}
?^n\left(\theta_1,\dots, \theta_n\right)=\left(\alpha_1,\dots, \alpha_n\right)
\in  \mathbf{R}^n/\mathbf{Z}^n,
\end{equation}
%******************************************************************************
where $0< \alpha_i< 1$ are irrational numbers.

\smallskip
(iii) Let us express the variable $T$ in terms of $?^n(x)$. Recall from item (iv) of Case I,  that
$T=k_i$.  On the other hand, one concludes from (\ref{eq3.9}),  that $2^{k_i}\sim 1/?^n\left(\frac{p_i}{q_i}\right)$. 
In other words,  we have:
%**********************************************************************************
\begin{equation}\label{eq3.18}
T=k_i \sim - \log_2\left( ?^n\left(\frac{p_i}{q_i}\right)\right). 
\end{equation}
%******************************************************************************

\smallskip
(iv) Let us observe that each $\theta_j$ in (\ref{eq3.17}) is the limit of 
a convergent sequence of the rational numbers $p_i/q_i$. Such a sequence
is uniquely defined by the Jacobi-Perron $n$-dimensional  continued fraction of the vector 
$(\theta_1,\dots,\theta_n)$  [Bernstein 1971] \cite{BE}.
We shall use formula $\log_2 N(V(k), T)\sim T^n$ established in Case I for $\theta_i=p_i/q_i$
to obtain the counting function for $\theta_k$ in (\ref{eq3.17}):
 %**********************************************************************************
\begin{equation}\label{eq3.19}
\log_2 N(V(k), T)=\lim_{i\to\infty} T^n\left(\frac{p_i}{q_i}\right). 
\end{equation}
%******************************************************************************

\smallskip
(v) We let $x=\lim p_i/q_i$ and using (\ref{eq3.18}) we exclude $T$ from (\ref{eq3.19}):
  %**********************************************************************************
\begin{equation}\label{eq3.20}
\log_2 N(V(k), T(x))\sim  (-1)^n\log_2^n\left(?^n\left(x\right)      \right). 
\end{equation}
%******************************************************************************
One can differentiate both sides of (\ref{eq3.20}) with respect to the variable $x$:
  %**********************************************************************************
\begin{equation}\label{eq3.21}
\begin{array}{lll}
\frac{d}{dx}\left[\log_2 N(V(k), T(x))\right]&\sim & (-1)^n \frac{d}{dx}\left[\log_2^n\left(?^n\left(x\right)\right)\right]=  \\
&&\\
=\frac{(-1)^nn}{\log 2} ~\frac{\log_2^{n-1}(?^n(x))}{?^n(x)} ~\frac{d ?^n(x)}{dx}&=&0.      
\end{array}
\end{equation}
%******************************************************************************
The last step of (\ref{eq3.21}) follows from the  differentiability of  $?^n(x)$, such  that the derivative is zero 
almost everywhere in $\mathbf{R}^n/\mathbf{Z}^n$, see  [Dushistova \& Moshchevitin 2010] \cite[Theorem 2, item 1]{DusMos1}
for $n=1$ and  [Panti 2008] \cite[Section 3]{Pan1} for $n\ge 2$. 
In particular,  we have $\frac{d?^n(x)}{dx} \equiv0$ at all rational points of  $\mathbf{R}^n/\mathbf{Z}^n$.
 Thus  the RHS of  (\ref{eq3.21}) is  zero and the function $\log_2^n\left(?^n\left(x\right)      \right)$ itself
is constant.  Using relation (\ref{eq3.20}),  we conclude that:
  %**********************************************************************************
\begin{equation}\label{eq3.22}
\log_2 N(V(k), T)\sim  Const. 
\end{equation}
%****************************************************************************** 

\medskip
Case III of Lemma \ref{lm3.3} is proved. 

\bigskip
All cases are considered and  Lemma \ref{lm3.3} is proved. 
\end{proof}
%*********************************************

\bigskip
Theorem \ref{thm1.1} follows from Lemma \ref{lm3.3}.

%**************************************************************************
\subsection{Proof of corollary \ref{cor1.2}}
%***************************************************************************
In outline, the proof is a translation of the data $rk ~K_0(\mathcal{A}_V) > n+1$
from algebraic to topological $K$-theory to the rational cohomology of $V$ using  the 
 Chern character formula. Let us pass to a detailed argument.

\medskip
(i)  Recall that the Serre $C^*$-algebra $\mathcal{A}_V$  is isomorphic to a crossed product $C(V)\rtimes_{\sigma}\mathbf{Z}$,
where $C(V)$ is the commutative $C^*$-algebra of complex valued functions on $V$ and $\sigma$ is an automorphism of $C(V)$
induced by such of the ring $R$, see Section 2.3 for the notation and details.   Thus one gets an isomorphism of the abelian groups:
  %**********************************************************************************
\begin{equation}\label{eq3.23}
K_0(\mathcal{A}_V)\cong K_0\left(C(V)\rtimes_{\sigma}\mathbf{Z}\right). 
\end{equation}
%****************************************************************************** 

\smallskip
(ii) On the other hand,  the Pimsner-Voiculescu exact sequence
 [Blackadar 1986] \cite[Theorem 10.2.1]{B}
implies an isomorphism of the $K$-groups: 
  %**********************************************************************************
\begin{equation}\label{eq3.24}
 K_0\left(C(V)\rtimes_{\sigma}\mathbf{Z}\right)\cong  K_1\left(C(V)\rtimes_{\sigma}\mathbf{Z}\right). 
\end{equation}
%****************************************************************************** 

\smallskip
(iii)  The same exact sequence of the abelian groups  
gives us an injective group homomorphism $i_*: K_1(C(V))\to  K_1\left(C(V)\rtimes_{\sigma}\mathbf{Z}\right)$
[Blackadar 1986] \cite[Theorem 10.2.1]{B}.  In particular, since  $K_1(C(V))$  is a subgroup of  $K_1\left(C(V)\rtimes_{\sigma}\mathbf{Z}\right)$,
 one concludes using formulas (\ref{eq3.23}) and (\ref{eq3.24}  that: 
  %**********************************************************************************
\begin{equation}\label{eq3.25}
rk ~K_1(C(V))\le rk ~K_0(\mathcal{A}_V). 
\end{equation}
%****************************************************************************** 

\smallskip
(iv)  We shall use an isomorphism formula linking the algebraic $K$-theory with the topological
$K$-theory. Namely, the Serre-Swan Theorem  says that:
  %**********************************************************************************
\begin{equation}\label{eq3.26}
K_1^{alg}\left(C(V)\right)\cong K^1_{top}(V). 
\end{equation}
%****************************************************************************** 

 \smallskip
 (v)  Finally,  the Chern character formula 
 recasts the topological $K^1$-group in terms of the odd rational cohomology 
 groups  of the variety $V$   [Blackadar 1986] \cite[Theorem 1.6.6]{B}:
   %**********************************************************************************
\begin{equation}\label{eq3.27}
K^1_{top}(V)\otimes \mathbf{Q} \cong  \bigoplus_{i=1}^n H^{2i-1}(V; \mathbf{Q}). 
\end{equation}
%****************************************************************************** 
 In particular, it follows from (\ref{eq3.27}) that if $\beta_i$ is the $i$-th Betti 
 number of $V$, then:
 %**********************************************************************************
\begin{equation}\label{eq3.28}
rk ~K^1_{top}(V)=  \sum_{i=1}^n \beta_{2i-1}. 
\end{equation}
%****************************************************************************** 
  
\smallskip
(vi)  In view of formulas (\ref{eq3.25})-(\ref{eq3.28}), one gets an inequality: 
%**********************************************************************************
\begin{equation}\label{eq3.29}
\sum_{i=1}^n \beta_{2i-1} \le  rk ~K_0(\mathcal{A}_V). 
\end{equation}
%******************************************************************************   

\smallskip
(vii) By an assumption of Corollary \ref{cor1.2},  we have $\sum_{i=1}^n \beta_{2i-1}>n+1$. 
The inequality (\ref{eq3.29}) implies that  $rk ~K_0(\mathcal{A}_V) > n+1$. 
But Theorem \ref{thm1.1} says the counting function is constant in this case.
The latter happens  if and only if $V(k)$ is a finite set. 

\bigskip
Corollary \ref{cor1.2} is proved.

%**************************************************************************
\section{Algebraic curves}
%***************************************************************************
We conclude by an example illustrating Theorem \ref{thm1.1} and Corollary \ref{cor1.2}.
For the sake of brevity, let us assume $n=1$. 
In other words, the variety $V(k)\cong X_g$ is an algebraic curve  of genus $g\ge 0$.
It is well known that the Betti numbers of $X_g$ are $\beta_0=\beta_2=1$ and $\beta_1=2g$. 
Following Theorem \ref{thm1.1},  consider the following three cases.

\bigskip
{\bf Case I: $rk ~K_0(\mathcal{A}_{X_g}) < 2$}. 
Let us calculate the value of $g\ge 0$ corresponding to this case. 
From (\ref{eq3.29}) one gets a lower bound for the rank:
%*********************************************************
\begin{equation}\label{eq4.1}
2g\le rk ~K_0(\mathcal{A}_{X_g}) < 2. 
\end{equation}
%******************************************************
Given that the rank is a positive  integer and $g\ge 0$,
we conclude that $g=0$ is the only value of $g$ satisfying
the double inequality (\ref{eq4.1}). Thus $X_{g=0}$ is the algebraic 
curve of genus zero and the base $2$ logarithm of the
counting function is given by the formula:   
%*********************************************************
\begin{equation}\label{eq4.2}
\log_2 N(X_{g=0}, T)\sim T.
\end{equation}
%****************************************************** 

\bigskip
{\bf Case II: $rk ~K_0(\mathcal{A}_{X_g}) = 2$}. 
Let us calculate  $g\ge 0$ corresponding to this case. 
Again, from (\ref{eq3.29}) one gets a lower bound for the rank:
%*********************************************************
\begin{equation}\label{eq4.3}
2g\le rk ~K_0(\mathcal{A}_{X_g}) =2. 
\end{equation}
%******************************************************
It is easy to see, that $g=0$ or $1$ are 
 the only values of $g$ satisfying
the double inequality (\ref{eq4.3}). 
However, the value $g=0$
is precluded by the Case I. 
Thus $X_{g=1}$ is the algebraic 
curve of genus one  and the base $2$ logarithm of the
counting function is given by the formula:   
%*********************************************************
\begin{equation}\label{eq4.4}
\log_2 N(X_{g=1}, ~T)\sim \log_2 ~T.
\end{equation}
%****************************************************** 

\bigskip
{\bf Case III: $rk ~K_0(\mathcal{A}_{X_g}) > 2$}. 
To determine the value of genus $g$ in this case, one can apply Corollary \ref{cor1.2}. 
Indeed, assuming $n=1$ one gets the inequality:
 %*********************************************************
\begin{equation}\label{eq4.5}
\beta_1=2g> 2.
\end{equation}
%****************************************************** 
 Thus $g\ge 2$ and the counting function is given by the formula: 
 %*********************************************************
\begin{equation}\label{eq4.6}
\log_2 N(X_{g\ge 2}, ~T)\sim Const.
\end{equation}
%****************************************************** 

\bigskip
Bringing together Cases I-III, one gets the following result. 
%**************************************************
\begin{corollary}\label{cor4.1}
%********************************************************************************************
\begin{equation}\label{eq4.7}
\log_2 ~N(X_g, T)\sim\begin{cases}
T, & \hbox{if} \quad g=0\cr 
\log_2 ~T, & \hbox{if} \quad g=1\cr
Const, & \hbox{if} \quad g\ge 2.
\end{cases}
\end{equation}
%***********************************************************************
 \end{corollary}
%***************************************************

\bigskip
%****************************************************
\begin{remark}
The reader can verify, that the counting function (\ref{eq4.7}) 
belongs to the family  (\ref{eq2.4}),  
where  $\mathcal{N}(X_g, T)=\log_2 ~N(X_g, T)$ and $a=b=1$. 
\end{remark}
%****************************************************

  %%%%%%%%%%%%%%
%\section*{Data availability}
%%%%%%%%%%%%%%%
  
  % Data sharing not applicable to this article as no datasets were generated or analyzed during the current study.
   
%%%%%%%%%%%%%%
%\section*{Conflict of interest}
%%%%%%%%%%%%%%%
%On behalf of all co-authors, the corresponding author states that there is no conflict of interest.
  
  %%%%%%%%%%%%%%

\bibliographystyle{amsplain}

\begin{thebibliography}{99}

\bibitem{BE}
L.~Bernstein, \textit{The Jacobi-Perron Algorithm, its Theory and Applications},
Lect. Notes in Math. {\bf 207}, Springer 1971. 



\bibitem{B}
 B.~Blackadar, \textit{$K$-Theory for Operator Algebras}, MSRI Publications,
 Springer, 1986.

\bibitem{DusMos1}
A.~A.~Dushistova and N.~G. ~Moshchevitin,
{\cyr  \textit{O proizvodnoi0  funktsii Minkovskogo $?(x)$}}, 
{\cyr Fundament. i Prikl. Matem.} {\bf 16} (2010),  33-44.

\bibitem{HS}
M.~Hindry and J.~H.~Silverman,
\textit{Diophantine geometry.
An introduction}, GTM {\bf 201},  Springer-Verlag, New York, 2000.

\bibitem{Min1}
H.~Minkowski, \textit{Zur Geometrie der Zahlen}, Verhandlungen 
des III  internationalen Mathematiker-Kongresses in Heidelberg,  pp. 164-173,  
Berlin, 1904.

 \bibitem{N}
I.~V.~Nikolaev, \textit{Noncommutative Geometry}, Second Edition,
De Gruyter Studies in Math. {\bf 66}, Berlin, 2022.

\bibitem{Pan1}
G.~Panti, 
\textit{Multidimensional continued fractions and a Minkowski function}, 
Monatsh. Math. {\bf 154} (2008),  247-264.

\bibitem{StaVdb1}
J.~T.~Stafford and M.~van ~den ~Bergh, \textit{Noncommutative curves and noncommutative
surfaces}, Bull. Amer. Math. Soc. {\bf 38} (2001), 171-216. 

\end{thebibliography}

%**********************************************************

\end{document}